\documentclass[a4,reqno,11pt]{amsart}
\usepackage{amsmath, amsfonts, amssymb, amsthm, amscd}
\usepackage[pdftex]{graphicx}
\usepackage[pdftex,dvipsnames]{xcolor}
\usepackage{mathtools}
\usepackage[percent]{overpic}

\usepackage[colorlinks]{hyperref}
\DeclareGraphicsExtensions{.ps,.eps}

\usepackage[latin1]{inputenc}
\usepackage{bbm}

\numberwithin{equation}{section}

\newtheorem{theorem}{Theorem}[section]

\newtheorem{lemma}{Lemma}[section]
\newtheorem{proposition}{Proposition}[section]

\newtheorem{definition}{Definition}

\newcommand{\suptwo}[2]{\sup_{\substack{#1 \\ #2}}} 
\newcommand{\abs}[1]{\left| #1\right|}

\newcommand{\calA}{\mathcal{A}}
\newcommand{\calB}{\mathcal{B}}

\newcommand{\calF}{\mathcal{F}}
\newcommand{\calG}{\mathcal{G}}

\newcommand{\calI}{\mathcal{I}}

\newcommand{\calS}{\mathcal{S}}
\newcommand{\calT}{\mathcal{T}}

\newcommand{\calZ}{\mathcal{Z}}

\newcommand{\bbA}{\mathbb{A}}

\newcommand{\bbE}{\mathbb{E}}

\newcommand{\bbN}{\mathbb{N}}

\newcommand{\bbP}{\mathbb{P}}

\newcommand{\bbR}{\mathbb{R}}

\newcommand{\bbZ}{\mathbb{Z}}

\newcommand{\sfc}{\mathsf c}

\newcommand{\sfe}{\mathsf e}

\newcommand{\sfi}{\mathsf i}

\newcommand{\sfC}{\mathsf{C}}

\newcommand{\uc}{\underline{c}}

\newcommand{\uh}{\underline{h}}

\newcommand{\uv}{\underline{v}}
\newcommand{\uu}{\underline{u}}
\newcommand{\ux}{\underline{x}}
\newcommand{\uy}{\underline{y}}
\newcommand{\uz}{\underline{z}}
\newcommand{\uw}{\underline{w}}

\newcommand{\ueta}{\underline{\eta}}
\newcommand{\unu}{\underline{\nu}}
\newcommand{\urho}{\underline{\rho}}
\newcommand{\ukappa}{\underline{\kappa}}
\newcommand{\uxi}{\underline{\xi}}
\newcommand{\uzeta}{\underline{\zeta}}
\newcommand{\uchi}{\underline{\chi}}

\newcommand{\uX}{\underline{X}}
\newcommand{\uY}{\underline{Y}}

\newcommand{\lb}{\left(}
\newcommand{\rb}{\right)}
\newcommand{\lbr}{\left\{}
\newcommand{\rbr}{\right\}}

\newcommand{\dd}{{\rm d}}

\newcommand{\step}[1]{S{\small TEP}\,#1.}

\newcommand{\1}{\mathbbm{1}}

\newcommand{\be}[1]{\begin{equation}\label{#1}}
\newcommand{\ee}{\end{equation}}

\newcommand{\ep}{\varepsilon}

\newcommand{\fkg}{\stackrel{\mathsf{FKG}}{\prec}}

\definecolor{darkblue}{rgb}{0,0.3,0.9}

\begin{document}
\date{\today} 

\title[Polymers under Geometric Area Tilts ]
{Tightness and Line Ensembles  for  Brownian Polymers under Geometric Area Tilts}

\author{Pietro Caputo}
\address{Dipartimento di Matematica e Fisica, Roma Tre University, Rome, Italy}
\email{caputo@mat.uniroma3.it}

\author{Dmitry Ioffe}
\address{Faculty of IE\&M, Technion, Haifa 32000, Israel}
\email{ieioffe@ie.technion.ac.il}
\thanks{DI was supported by the Israeli Science Foundation grants 
1723/14 and 765/18.}

\author{Vitali Wachtel}
\address{Institut f\"ur Mathematik, Universit\"at Augsburg, D-86135 Augsburg, 
Germany}
\email{vitali.wachtel@math.uni-augsburg.de}
\thanks{VW was supported by the Humboldt Foundation.}

\begin{abstract}
We prove tightness and limiting Brownian-Gibbs 
description for line ensembles  of non-colliding Brownian bridges  above a hard wall, 
which 
are subject to geometrically growing self-potentials of tilted area type. Statistical properties of the resulting ensemble are very different from that of non-colliding Brownian bridges without self-potentials. 
The model itself was  introduced in order to mimic level lines of $2+1$  discrete Solid-On-Solid 
 random interfaces above a hard wall. 
\end{abstract}

\maketitle
\centerline{{\em We dedicate this paper to Anton Bovier  on the occasion of his 60th birthday}}
\section{Introduction, Notation and Results 
} 

\subsection{\bf Level lines of SOS surfaces and related ensembles of non-intersecting paths} 
{
	In this paper we investigate large scale behavior of ensembles of
	non-intersecting Brownian polymers under geometric area tilts. 
	Our main motivation comes from a desire to understand and model 
	limiting properties of microscopic large level lines of  the 
	$2+1$ SOS (solid-on-solid) random interfaces conditioned to stay above 
	a flat wall. We refer to \cite{CIW18} for a careful exposition of this connection with references to earlier works.
} 

{
	Models of a (single) Brownian polymer
	constrained to stay above a barrier
	were introduced in \cite{ferrarispohn2005}. By Girsanov's transform an equivalent formulation is in terms of positive Brownian bridges under area tilts~\cite{pascalzeitouni}. These models are  
	of an independent interest and, in particular, they have a rich limiting variational and fluctuation structure. We would like to mention recent 
	works \cite{smith2018geometrical,dobrokhotov2018asymptotics,meerson2019geometrical} 
	and references therein. 
} 

{
	Gibbsian structure of (single $\sfC (\bbR , \bbR)$-valued)  Brownian polymers with self-potentials and self-interactions was introduced and explored in \cite{osada1999gibbs,lHorinczi2001gibbs}. Space-time Gibbsian states on ${\sfC (\bbR , \bbR)}^{\bbZ^d}$ were introduced and constructed in \cite{minlos2000gibbs,dai2002existence} in the weak interaction case, both as Gibbs states and as  systems  
	of infinite dimensional interacting diffusions. Hard core interactions between ordered paths in one spatial dimension 
	do not fall into this framework. Variants of infinite volume dynamics 
	whose equilibrium states are determinantal point fields were constructed in \cite{spohn1987interacting,osada2004non}, see 
	also various refinements with references to the latter works.
} 

{
	Recent developments which are closely related to the model we consider here are \cite{Bornemann,duits2018}. However, 
	since geometrically growing area  tilts preclude using Karlin-McGregor formula, the determinantal structure is apparently lost, and one has to deal with a rather different situation. For instance no rescaling as the number  of paths grow is needed. Yet, 
	the Brownian-Gibbsian framework introduced in 
	\cite{corwinhammond,corwin2016kpz} fits in extremely well, and we shall {rely} on the philosophy and the ideas which were developed in these remarkable papers.  
}
\subsection{\bf Organization  of the paper.}
{
	The Brownian 
	notation and the models of Brownian polymers under geometric area tilts are introduced in Subsections~\ref{sec:BM-BP}-\ref{sub:pmeas}. The main confinement
	results from \cite{CIW18} which enables a uniform control over one-point distributions of the top path is recalled in 
	Subsection~\ref{sec:conf}. Our main results here which are formulated in Subsection~\ref{sub:Results} yield full tightness of the ensembles we consider, both with free and zero boundary conditions, and, furthermore, we establish Brownian-Gibbs property of limiting infinite dimensional ensembles.  Whether there is a unique such limiting ensemble or not remains an open question, although monotonicity arguments imply unicity of the limit for
	polymer measures with zero boundary conditions. The proofs of
	full tightness appear in Section~\ref{sec:scheme}. Limiting line
	ensembles are discussed  in  Section~\ref{sec:line-ens}. {Finally, future research directions and some  open problems are outlined in the concluding Subsection~\ref{sub:open}.   
	}

\subsection{\bf 
	Brownian motion and Brownian bridges.}  
\label{sec:BM-BP}
In the sequel we shall use the same notation for path measures of underlying Brownian motion and 
Brownian bridges and for expectations with respect to these
path measures. 
For  $\ell < r$ and $x\in\bbR$, 
let ${\mathbf P}^x_{\ell , r} $ be the  path measure of the Brownian 
motion $X$ on $[\ell , r]$ which starts at $x$ at time $\ell$; $X (\ell ) = x$. 
We can record ${\mathbf P}^x_{\ell , r} $ as follows: 
\be{eq:BM-BB} 
{\mathbf P}^x_{\ell , r} \lb F (X ) \rb = \int {\mathbf B}^{x, y}_{\ell , r} \lb F ( X )\rb \dd y
\ee
where ${\mathbf B}^{x,y}_{\ell , r} $ the  {\em unnormalized} path measure of the Brownian 
bridge  $X$ on $[\ell , r]$ which starts at $x$ at time $\ell $ and ends at 
$y$ at time $r$; $X (\ell  ) = x,\, X (r )= y$.  
%
{
In this way {the total mass of ${\mathbf B}^{x,y}_{\ell , r}$ is given by}
\be{eq:q-xy} 
{
q_{r-\ell} (x ,y ) := 	
}
{\mathbf B}^{x,y}_{\ell , r}  (1 ) = \tfrac1{\sqrt{2\pi(r-\ell )}}\,{\rm e}^{- \frac{(y-x)^2}{2(r-\ell )}}. 
\ee
}
{ 
The corresponding normalized Brownian bridge probabilities will be denoted
as ${\bf \Gamma}^{x,y}_{\ell , r}$, that is ${\bf \Gamma}^{x,y}_{\ell , r} (\cdot )= 
{\mathbf B}^{x,y}_{\ell , r}(\cdot )/q_{r-\ell} (x ,y )$.  
}
\smallskip 

For an 
$n$-tuple {$\ux=(x_1,x_2,\ldots,x_n)\in\bbR^n$}, set 
\[
{\mathbf P}^{\ux}_{n; \ell , r} 
= 
{\mathbf P}^{x_1}_{\ell , r} \otimes {\mathbf P}^{x_2}_{\ell , r} \otimes \cdots \otimes 
{\mathbf P}^{x_n}_{\ell , r} .
\]
Similarly for 
$n$-tuples $\ux, \uy \in\bbR^n$, set 
\be{eq:q-uxuy01}
{\mathbf B}^{\ux , \uy}_{n; \ell , r} 
=
{\mathbf B}^{x_1, y_1 }_{\ell , r} \otimes {\mathbf B}^{x_2 , y_2}_{\ell , r} 
\otimes \cdots \otimes 
{\mathbf B}^{x_n , y_n }_{\ell , r}, 
\ee
and   
\be{eq:q-uxuy02}
q_{n,r-\ell}(\ux , \uy) = \prod_{i=1}^{n} 
q_{n,r-\ell} (x_i , y_i )\,,\qquad 
{\bf \Gamma}^{\ux , \uy}_{n,\ell , r} = \frac{{\mathbf B}^{\ux , \uy}_{\ell , r}}{q_{n,r-\ell}(\ux , \uy)}	.
\ee

\subsection{\bf Polymer measures  with geometric area tilts.}
\label{sec:geopol}
Given a function $h$, the signed $h$-area under the trajectory of $X$ is defined as
\be{eq:T-area} 
\calA_{\ell , r}^h \lb X\rb = \int_{\ell}^r h (t ) X (t)\dd t .
\ee
We shall drop the superscript if $h\equiv 1$ and use $\calA_{\ell , r} \lb X\rb$ accordingly. 
For  $n\in\bbN$ define
\be{eq:Aplus}
\bbA_n^+ (M) = \{ \ux \in \bbR^n \,:\, M > x_1 >\dots > x_n > 0\} \quad\text{and}\quad \bbA_n^+ = \bbA_n^+ (\infty ). 
\ee
We also define the set
\be{eq:Aplusbar}
\bar\bbA_n^+ = \{ \ux \in \bbR^n \,:\,  x_1 \geq \dots \geq x_n \geq  0\}. 
\ee
Polymer measures which we consider in the sequel are always  
concentrated on the set $\Omega^{+}_{n;  \ell, r}$ of  $n$-tuples  $\uX$, 
\be{eq:Omega-Set} 
\Omega^{+}_{n; \ell , r} = \lbr \uX~:~ \uX (t )\in \bbA_n^+\ \ \forall\, t\in (\ell , r)\rbr.
\ee
Given 
 $n\geq 1$,  $a>0,\lambda >1$ and $\ux,\uy\in\bbA_n^+$, consider the
partition 
function
\be{eq:PF-BC-T01} 
 Z_{n; \ell , r}^{ \ux , \uy } (a , \lambda ) 
:=  {\mathbf B}^{\ux , \uy}_{\ell , r} 
\lb 
\1_{\Omega_{n; \ell , r}^+}
{\rm e}^{-\sum_1^n a\lambda^{i-1}\calA_{\ell , r} (X_i )}
\rb, 
\ee
and the associated probability measure $\bbP^{\ux,\uy}_{n; \ell , r}\left[ \cdot  ~|a, \lambda  \right]$ defined by
\be{eq:PF-BC-T02} 
\bbP^{\ux,\uy}_{n; \ell , r}
\left[ F (\uX ) ~|a, \lambda  \right]  
:= \frac{1}{Z_{n; \ell , r}^{ \ux , \uy } (a , \lambda ) } \,
{\mathbf B}^{\ux , \uy}_{\ell , r} 
\lb F (\uX )
\1_{{\Omega_{n; \ell, r }^+}}
{\rm e}^{-\sum_1^n a\lambda^{i-1}\calA_{\ell , r} (X_i )} 
\rb ,
\ee
where $F$ is any bounded measurable function over the set of $n$-tuples of continuous functions from $[\ell,r]$ to $\bbR$. 
The measure $\bbP^{\ux,\uy}_{n; \ell , r}\left[ \cdot  ~|a, \lambda  \right]$ will be referred to as the $n$-{\em polymer measure with 
 $(a,\lambda)$-geometric area tilts with boundary conditions $(\ux,\uy)$ on the interval} $[\ell,r]$.  

We remark that $\bbP^{\ux,\uy}_{n; \ell , r}\left[ \cdot  ~|a, \lambda  \right]$ is well defined for all $\ux,\uy\in\bbA_n^+$. Indeed, if $\ux,\uy\in\bbA_n^+$
one has 
\be{eq:PF-BC-T03} 
 {\mathbf B}^{\ux , \uy}_{\ell , r} 
\lb 
{\Omega_{n; \ell , r}^+}
\rb >0 , 
\ee
 and therefore $Z_{n; \ell , r}^{ \ux , \uy } (a , \lambda )\in (0,\infty)$. While this does not apply to all $\ux,\uy\in\bar \bbA_n^+$, it is still possible to define $\bbP^{\ux,\uy}_{n; \ell , r}\left[ \cdot  ~|a, \lambda  \right]$ in these cases by a limiting procedure, see e.g.\ \cite[Definition 2.13]{corwinhammond} for the case $a=0$. In particular, we shall often deal with the case of {\em zero boundary conditions} $$\bbP^0_{n; \ell , r}\left[ \cdot  ~|a, \lambda  \right]:=\bbP^{\underline 0,\underline 0}_{n; \ell , r}\left[ \cdot  ~|a, \lambda  \right].$$
It is also natural to consider the polymer measure with {\em free boundary conditions} defined by  
\be{eq:PolMeas02} 
\bbP_{n; \ell , r}
\left[ F (\uX ) ~|a, \lambda  \right]  
:= \frac{1}{{\mathcal Z}_{n ;\ell , r} (a, \lambda )} 
\int_{\bbA_{n}^+} 
\int_{\bbA_{n}^+}
{\mathbf B}^{\ux , \uy}_{\ell , r} 
\lb F (\uX )
\1_{\Omega_{n, T }^+}
{\rm e}^{-\sum_1^n a\lambda^{i-1}\calA_{\ell , r} (X_i )} 
\rb 
\dd \ux \dd\uy 
\ee
where 
\be{eq:freepartfct}
{\mathcal Z}_{n ;\ell , r} (a, \lambda )
:= 
\int_{\bbA_{n}^+} 
\int_{\bbA_{n}^+}
Z_{n ;\ell , r}^{ \ux , \uy}  (a, \lambda)  
\dd\ux 
\dd 
\uy ,
\ee
and $\dd\ux, \dd 
\uy$ denote Lebesgue measures on $\bbR^n$.
For a proof that $\bbP_{n; \ell , r}
\left[ \cdot~|a, \lambda  \right] $ is well defined, that is ${\mathcal Z}_{n ;\ell , r} (a, \lambda )\in(0,\infty)$, for all $a>0,\lambda>1$, see \cite[Appendix~A]{CIW18}.

\subsection{\bf A general class of polymers with area tilts.}
\label{sub:pmeas}
Let us say that two 
functions $f$ and $g$ on ${I\subset \bbR}$ satisfy $f\prec g$ if $ f (t )\leq g (t )$ 
for any $t\in I$. By construction, 
if $\uX\in \Omega_{n ;\ell , r }^+$, then $0\prec X_n\prec  X_{n-1} \prec \dots \prec X_1$ {on $[l,r]$}. For every $n\in \bbN$ 
and $\ell < r$, consider the following 
general class $\bbP_{n ;\ell , r}^{\ux , \uy}\left[\, \cdot | h_- , h_+, \urho\right]$ of
polymer measures which is parametrized by: 
\begin{description} 
 \item[a] Boundary conditions 
 $\ux, \uy \in \bbA_n^+$. 
 \item[b] A pair $\uh=(h_- , h_+)$ of non-negative continuous functions, called respectively  
 the floor and the ceiling, satisfying $h_-\prec h_+$ on $[ \ell , r]$. 
 \item[c] An $n$-tuple of (not necessarily ordered) nonnegative 
 continuous functions  $\urho = \lbr \rho_1, \dots , \rho_n\rbr$, 
 called the area tilts. 
\end{description}
Then, setting 
 \be{eq:Om-h} 
\Omega_{n ;\ell , r }^{\uh} = \Omega_{n; \ell , r }^+
{\cap \{  h_- \prec X_n\} \cap \{ X_1 \prec h_+\}}, 
\ee 
define:
 \be{eq:mu-meas} 
 \bbP_{n ;\ell , r}^{\ux , \uy}\left[\dd\uX  | h_- , h_+, \urho\right] \,
 \propto \, 
 {\rm e}^{- \sum_1^n \calA_{
 {\ell , r}}^{\rho_i} (X_i )} \1_{\Omega_{n ;\ell , r }^{\uh}}
 {\mathbf B}_{\ell , r}^{\ux , \uy }\lb \dd \uX \rb .
 \ee
The corresponding partition function is denoted 
$Z_{n ;\ell , r}^{\ux , \uy}( h_- , h_+, \urho )$. Clearly, the polymer measure $\bbP_{n ;\ell , r}^{\ux , \uy}\left[\cdot  | h_- , h_+, \urho\right]$ coincides with the $\bbP^{\ux,\uy}_{n; \ell , r}\left[ \cdot  ~|a, \lambda  \right]$  defined in \eqref{eq:PF-BC-T02}  in the case of geometric {tilts} $\rho_i\equiv a\lambda ^{i-1}$ and trivial floor and ceilings $(h_-,h_+)\equiv (0,+\infty)$.
Similarly, we shall employ the following notation for 
	Brownian bridge measures conditioned to $\Omega_{n; \ell ,r }^{\uh}$,  
	\be{eq:B-y-meas} 
	{\mathbf B}_{n; \ell , r}^{\ux , \uy } \left[  \, \cdot\, \big|\, 
	h_- , h_+
	\right] 
	:= 
	{{\bf \Gamma}}_{n; \ell , r}^{\ux , \uy } \left[  \, \cdot\, \big|\, 
	h_- , h_+
	\right] 
	:= 
	{{\bf \Gamma}}_{n; \ell , r}^{\ux , \uy } \lb \, \cdot\, \big|\, \Omega_{n; \ell ,r }^{\uh} \rb .
	\ee 
{
There is a straightforward extension of  \eqref{eq:Omega-Set}
and \eqref{eq:Om-h} to sets of paths over more general subsets $I\subset \bbR$ (in the sequel we shall need to work with  finite union of intervals). Namely, 
\be{eq:Omega-Set-g} 
\Omega^{+}_{n; I} = \lbr \uX~:~ \uX (t )\in \bbA_n^+\ \ \forall\, t\in I\rbr
\ee	
and 
\be{eq:Om-h-g} 
\Omega_{n ; I  }^{\uh} = \Omega_{n; I }^+\cap
{\{  h_- \prec X_n\} \cap \{ X_1 \prec h_+\}}
\ee 
}

Moreover, there is a  version of \eqref{eq:mu-meas} which permits  more general 
boundary conditions: Let $\unu$ and $\ueta$ be $n$-tuples of 
 functions on $\bbR_+$. 
For $\ux\in \bbA_n^+$ set $\unu (\ux ) =\sum_1^n \nu_i (x_i )$, and $\ueta (\ux ) =\sum_1^n \eta_i (x_i )$. 
Similarly, set $\calA_{
{\ell , r}
}^{\urho} (\uX ) = \sum_1^n \calA_{
{\ell , r}}^{\rho_i} (X_i )$. 
Then, 
\be{eq:mu-meas-bc} 
 \bbP_{n;  \ell , r}^{\unu , \ueta}\left[\dd\uX  | h_- , h_+, \urho\right] \,
 \propto \, 
 \int_{\bbA_n^+} \int_{\bbA_n^+} {\rm e}^{- \calA_T^{\urho} (\uX )} \1_{\Omega_{n; \ell ,r }^{\uh}}\lb \uX\rb 
 {\rm e}^{-\unu (\ux )}{\mathbf B}_{\ell , r}^{\ux , \uy }\lb \dd \uX \rb{\rm e}^{-\ueta (\uy )}\dd\ux \dd \uy .
 \ee
The corresponding partition function is denoted 
$\calZ_{n; \ell , r}^{\ueta , \unu}( h_- , h_+, \urho )$. 

Below we shall tacitly assume that boundary conditions $\unu , \ueta$ are chosen in such a 
way that the corresponding polymer measures are well defined, this is 
justified in  all the relevant cases 
in  Appendix~A of \cite{CIW18}.

\bigskip

\noindent
{\em Reduced notation.}
We shall, unless this creates a confusion, employ the following reduced
notation:
If $\ueta,\unu$ are identically zero, 
we shall drop them  from the notation. We refer to this as the case of free (or empty) boundary conditions.  Similarly we shall drop from the notation the floor $h_-$ 
whenever $h_- \equiv 0$ and the ceiling $h_+$ whenever $h_+\equiv \infty$. 

In the case of pure geometric tilts we shall write $a, \lambda$ instead of $\urho$ whenever 
\be{eq:GAT}
\urho = \lbr a, a\lambda , \dots ,a\lambda^{n-1}, \dots \rbr
\ee 
and, furthermore,  wee shall drop $a, \lambda$ from the notation whenever this creates no confusion. 

{
	In the sequel we shall drop sub-index $n$ unless we would like  to stress the number of polymers in a stack, and  whenever this will
	cause no confusion.
}

In the case of symmetric time intervals  we shall write 
\[
\Omega^{+}_{n, T} =\Omega^{+}_{n; -T, T},\quad  
Z_{n, T}^{ \ux , \uy}   =   
Z_{n; -T  , T}^{ \ux , \uy}  , \quad 
{\mathbf B}_{n, T}^{ \ux , \uy}   =   
{\mathbf B}_{n; -T  , T}^{ \ux , \uy} , \quad 
 \bbP^{\ux,\uy}_{n, T} = \bbP^{\ux,\uy}_{n; -T, T},
\] 
and so on.

\subsection{\bf 
	A confinement statement and uniform control of curved maxima.} 
\label{sec:conf}
The main result of  \cite{CIW18} could be formulated as follows:  
\begin{theorem} 
\label{thm:main} 
For any fixed $a>0$, $\lambda >1$ and $\chi >0$
 the family of (one-dimensional)  
 distributions  of the height  $X_1 (0 )$  of the top path at the origin under the free boundary condition field 
 $\lbr \bbP_{n ;\ell , r}\left[ \cdot  ~|a, \lambda  \right]\rbr_{n\in\bbN , \ell\leq -\chi  , r\geq \chi }$ 
 is tight. In other words the top path does not fly away as the number of
 polymers and the length of their horizontal span grow. Moreover, the same statement holds for the zero-boundary condition field $\bbP^0_{n; \ell , r}\left[ \cdot  ~|a, \lambda  \right]$.
\end{theorem}

The statement of Theorem~\ref{thm:main} follows from uniform control of the expectations of a certain 
 curved maxima. Let us describe the latter notion. 
 
Let $\varphi_\alpha  (t ) = \abs{t}^\alpha$ with $\alpha \in (0, \frac{1}{2} )$. 
Given a continuous function $h$ on $[\ell , r]$ define (see Figure~\ref{fig:a}) 
\be{eq:xi-func} 
\xi^{\ell , r}_{\alpha } ( h) = \min\lbr y\geq 0\,  :\, y+\varphi_\alpha  \succ h\rbr = \max_{t\in[\ell ,r]} \left(  h (t ) - \abs{t}^\alpha  \right)_+, 
\ee
where $(\,\cdot\,)_+$ denotes the positive part. 
Informally, $\xi^{\ell , r}_{\alpha } ( h)$ is the minimal amount to lift $\varphi_\alpha $ so that 
it will stay above $h$. We think of $\xi^{\ell , r}_{\alpha } ( h)$ in terms of the curved maximum of $h$ on
$[ \ell , r]$.

\begin{figure}[h]
\begin{overpic}[scale=0.5]{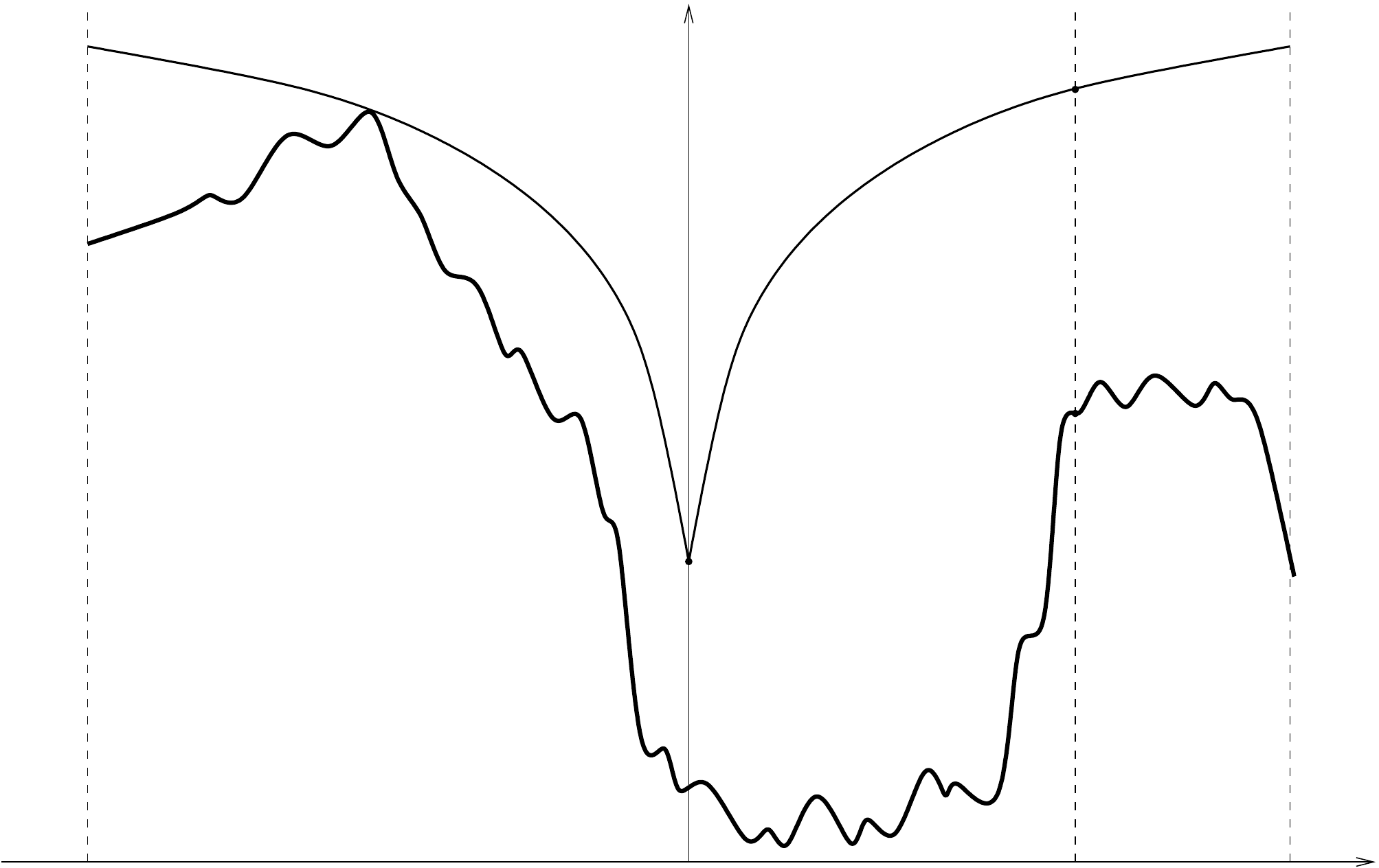}
\put(51,21){\scalebox{.7}{$\xi^{\ell,r}_\alpha(h)$}}
\put(62,57.5){\scalebox{.7}{$\xi^{\ell,r}_\alpha(h)+|t|^\alpha$}}
\put(72,34){\scalebox{.76}{$h(t)$}}
\put(77.5,-3){\scalebox{.8}{$t$}}
\put(92.5,-3){\scalebox{.8}{$r$}}
\put(48.9,-3){\scalebox{.8}{$0$}}
\put(4.9,-3){\scalebox{.8}{$\ell$}}
\end{overpic}
\label{fig:a}
\caption{ The curved maximum $\xi^{\ell,r}_{\alpha}(h)$, in the case where $[\ell,r]$ is a symmetric interval around the origin. 
} 
\vspace{-0.25cm}
\end{figure}
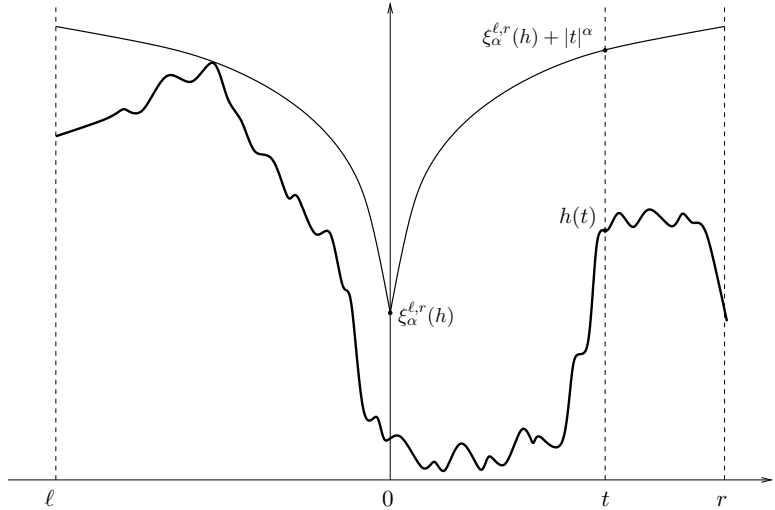

%

Then, Theorem~\ref{thm:main} is an immediate consequence of the following result \cite{CIW18}:
\begin{theorem} 
\label{thm:max-control}
For any $a>0,\lambda>1$, $\chi >0$ and $\alpha\in(0,\tfrac12)$:
\be{eq:max-control}
\begin{split}
\xi^*_\alpha  &:=  \sup_{\ell\leq -\chi} \sup_{ r \geq \chi }\sup_n \bbE_{n; \ell , r} \left[  \xi^{\ell , r}_{\alpha} ( X_1 )\, |\, 
 a, \lambda\right]  \\
&=
\sup_{\ell\leq -\chi} \sup_{ r \geq \chi }\sup_n \bbE_{n; \ell , r} \left[
\max_{t\in[\ell ,r]} \left( X_1 (t ) - \abs{t}^\alpha  \right)_+\, 
\big|\,  a, \lambda\right] 
 <\infty .
 \end{split}
 \ee
\end{theorem} 
The fact that Theorem \ref{thm:max-control}
also holds for the zero boundary condition field is a consequence of stochastic domination, see  Section \ref{sub:SD}  below. 
\subsection{Full tightness, non-intersecting line ensembles and Brownian-Gibbs property}
\label{sub:Results}
The main contribution of this paper is a strengthening of the above mentioned results, with the extension of the one-dimensional  tightness  to a full tightness statement on path space.  To formulate this we need to introduce some more notation.    
Fix $a >0$ and $\lambda >1$. Given $\gamma >0$ and $k\in \bbN$, 
 let  $n >k$, 
 $[\ell , r]\supseteq [-\gamma , \gamma ]$, and   
consider the free boundary condition field $\bbP_{n; \ell , r} [\, \cdot\, | a, \lambda ]
$. 
We shall 
use $\mu^{k , \gamma}_{n ;\ell , r}$ for the induced distribution of $k$ top paths
$\uX = \lb X_1 (\cdot ) , \dots , X_k (\cdot )\rb$ over the interval $[-\gamma , \gamma]$. Similarly, we write $\mu^{0,k , \gamma}_{n ;\ell , r}$ for the distribution of $k$ top paths under the zero boundary condition field $\bbP^0_{n; \ell , r} [\, \cdot\, | a, \lambda ]$. 
In this way $\mu^{k , \gamma}_{n ;\ell , r},\mu^{0,k , \gamma}_{n ;\ell , r}$ are distributions on the set of continuous functions from $[-\gamma,\gamma]$ to $\bar \bbA_k^+$, denoted $\sfC\lb [-\gamma , \gamma ]; \bar\bbA_k^+\rb$. The latter is equipped with the {topology of uniform convergence}. 
\begin{theorem} 
\label{thm:tight-k} 
For any $k\in \bbN$ and $\gamma >0$ the family 
\be{eq:fam-g}
\lbr \mu^{k , \gamma}_{n ;\ell , r}\,,\;\,[\ell , r]\supseteq [-\gamma , \gamma  ], n>k\rbr
\ee
 is tight on  $\sfC\lb [-\gamma , \gamma ]; \bar\bbA_k^+\rb$. The same holds for the family $\lbr \mu^{0,k , \gamma}_{n ;\ell , r}\,,\;\,[\ell , r]\supseteq [-\gamma , \gamma  ], n>k\rbr$.
	\end{theorem}

We turn to the analysis of the consequences of the above tightness results regarding convergence to a limiting polymer measure describing infinitely many non-intersecting random lines. We start by recalling the definition of Brownian-Gibbs line ensembles.  We refer to 
\cite{corwinhammond,corwin2016kpz,dauvergne2018basic} for a more
comprehensive setup, as well as to earlier 
papers 
\cite{minlos2000gibbs,osada2004non,osada1999gibbs} on the subject. Here we formulate the concepts that are relevant for our setting.  

Define the spaces
\[\bbA_\infty^+ = \lbr \ux\in (0, \infty)^\bbN\, :\, x_1 > x_2 >\dots >0\rbr\, \subset \, 
\bar{\bbA}_\infty^+ = \lbr \ux\in [0, \infty)^\bbN\, :\, x_1 \geq  x_2 \geq \dots \geq 0\rbr .
\] 
The sample space is   
$\Omega =\sfC\lb \bbR,\bar{\bbA}_\infty^+\rb$, the set of continuous functions $f:\bbR\mapsto\bar{\bbA}_\infty^+$, equipped with the topology of uniform convergence of any finite number of paths on compact subsets, and   with the corresponding Borel $\sigma$-field $\calB$. The coordinate maps 
$\uX\in \Omega \mapsto X_i (t )$ are viewed as position of $i$-th particle at time $t$.  
Following \cite{corwinhammond}, for each $n\in\bbN$, and time interval $[\ell,r]\subset\bbR$, define
the {internal} and external $\sigma$-algebras
\be{eq:Bext} 
{
\calB_{n; \ell , r}^{\sfi} 
= \sigma\lb X_i (t ) :\, \text{{$t{\in} (\ell , r)$ {and} $i \leq n$}}\rb
}
\ {\rm \;and\;}\ 
\calB_{n; \ell , r}^{\sfe} 
= \sigma\lb X_i (t ):\, \text{{either $t{\notin} (\ell , r)$ {or} $i >n$}}\rb. 
\ee
Recall that $\bbE_{n ;\ell , r}^{\ux , \uy} \left[\, \cdot\, 
\Big| h, a,\lambda\right]$ stands for the expectation w.r.t.\ $n$-polymer measure with 
 $(a,\lambda)$-geometric area tilts with boundary conditions $(\ux,\uy)$ and floor $h$ on the interval $[\ell,r]$. 

\begin{definition}\label{def:BG}
A probability measure $\bbP$ on $\Omega$ is said to have the 
{\em Brownian-Gibbs property} with respect to $(a , \lambda )$-geometric area tilts \eqref{eq:GAT} if 
for any bounded measurable $F:\Omega\mapsto \bbR$, 
\be{eq:BG-prop} 
\bbE\lb \, F\, \Big| \, \calB^{\sfe}_{n ;\ell , r} \rb 
= 
\bbE_{n ;\ell , r}^{\uX^{(\leq n)} (\ell) , \uX^{(\leq n)} (r )} \left[\,F(\cdot,X^{(>n)}) 
\Big| X_{n+1},a,\lambda\right], 
\ee
$\bbP$-a.s for any $-\infty <\ell <r <\infty$ and $n\in \bbN$.
In \eqref{eq:BG-prop}, we use the notation $\uX^{(\leq n)}=(X_1,\dots,X_n)$ and $\uX^{(>n)}=(X_{n+1},X_{n+2},\dots)$.  A probability measure $\bbP$ on $\Omega$ with the above Brownian-Gibbs property is called a Brownian-Gibbs measure or a Brownian-Gibbs line ensemble with respect to $(a , \lambda )$-geometric area tilts. The line ensemble is said to be {\em non-intersecting} if $\bbP$ is concentrated on $\sfC\lb \bbR,{\bbA}_\infty^+\rb\subset \sfC\lb \bbR,\bar{\bbA}_\infty^+\rb$. 
\end{definition}

Notice that measures $\bbP^{\ux,\uy}_{m ;S , T}$ and $\bbP_{m ;S , T}$ describing $m$ lines in the time interval $[S,T]$, as defined in Section \ref{sec:geopol}, are trivially extendable to the 
whole of $\Omega$  
by setting $X_i \equiv 0$ for $i >m$ and $\uX (t )\equiv \uX (T)$ for $t >T$ (respectively, $\uX (s )\equiv \uX (S)$ for $s <S$).
Moreover, these measures satisfy the  Brownian-Gibbs property \eqref{eq:BG-prop} whenever $S\leq \ell<r\leq T$ and $n\leq m$.

Theorem \ref{thm:tight-k} says that any sequence from $\{\mu^{k , \gamma}_{n ;\ell , r}\,,\;\,[\ell , r]\supseteq [-\gamma , \gamma  ], n>k\}$ has a weakly converging subsequence in the space of probability measures on $\sfC\lb [-\gamma , \gamma ]; \bar\bbA_k^+\rb$, for all fixed $k\in\bbN$ and $\gamma>0$. Moreover, by using an extraction argument together with consistency of marginals along $\gamma\to\infty$ and $k\to\infty$,  Theorem \ref{thm:tight-k} also implies that any sequence from $ \{\bbP_{n ;\ell , r}\}$ or from $\{ \bbP^0_{n ;\ell , r}\}$ has a subsequence that converges weakly to a probability measure $\bbP$ on $\sfC\lb \bbR,\bar{\bbA}_\infty^+\rb$. 
Concerning the nature of the limit points, the main consequences of our analysis can be summarized as follows. 
\begin{theorem} 
\label{thm:coarse1} 
Any sub-sequential limit (as $n\to\infty , \ell\to -\infty$ and 
$r\to\infty$)  $\bbP$ of  $ \{\bbP_{n ;\ell , r}\}$ or $\{ \bbP^0_{n ;\ell , r}\}$	
 has the following properties: 
 \begin{enumerate} 
 \item $\bbP$ is Brownian-Gibbs with respect to $(a , \lambda )$-geometric area tilts.
 \item $\bbP$ is non-intersecting. 
 
 \item Control of maxima: For any $\gamma>0$, there exists $C=C(\gamma)>0$ such that for all $k\in\bbN$, and $M>0$: 
\be{eq:contr_max} 
\bbP\lb \max_{t\in [-\gamma , \gamma]} X_{k}(t) >\lambda^{-(k-1)/3}M \rb \leq \frac{C}M\,.
\ee
 \item {Absolute continuity with respect to independent Brownian bridges: For any $k\in \bbN, M >0$ and $\gamma >0$ there exists 
 $\beta = \beta (k , M , \gamma )>0$ such that 
\be{eq:abs-cont} 
\bbP \lb E \rb \leq \frac{ C( \gamma )}{M} + 
\frac{1}{\beta} 
\sup_{\uv, \uw \in \bbA_n^+ ( M) }
{{\bf \Gamma}}_{{k, \gamma}}^{\uv, \uw} 
\lb  E\rb , 
\ee
for any $E\in \calB_{k; \gamma}^{\sfi}$. 
 }
 	\end{enumerate}
	\end{theorem}
	Finally, let us remark that Theorem~\ref{thm:coarse1} does not imply unicity of limiting points, and therefore it does not allow us to conclude that $\bbP_{n ;\ell , r}$ or $\bbP^0_{n ;\ell , r}$ actually converge as $n\to\infty , \ell\to -\infty$ and 
$r\to\infty$. 
{
Moreover, Theorem~\ref{thm:coarse1} does not address the issue of ergodicity or even time stationarity of limiting points.  Neither it
addresses the issue of  representation of limiting points as infinite dimensional  diffusions. 
}

 However, thanks to monotonicity,  we shall  obtain the desired convergence {and time stationarity}, at least for the zero-boundary condition field. We believe that the same holds for free boundary conditions, and that the limiting ensembles coincide. 
\begin{theorem} 
	\label{thm:coarse2} 
	There is a unique (as $n\to\infty , \ell\to -\infty$ and 
	$r\to\infty$) limit $\bbP^{0}$ of 
	the family $\bbP^0_{n ;\ell , r}$. Besides the properties 
	listed in Theorem~\ref{thm:coarse1},  $\bbP^0$  is {time stationary}. 
\end{theorem} 

\section{Tightness of ensembles of top paths}\label{sec:scheme}
We follow the scheme developed in the seminal \cite{corwinhammond}. Since, however, 
geometric area tilts preclude exact solutions and, somehow, complicate matters, we had to
modify, adjust and eventually simplify  the arguments involved. 
\subsection{\bf Stochastic domination.} 
\label{sub:SD} 
As in \cite{corwinhammond,CIW18} stochastic domination plays an important technical role in our approach. 
Equip $\Omega_{n; \ell , r }^{+}$ with the partial  order $\prec $, defined by 
$$
\uX\prec \uY\;\;\text{iff}\;\; X_i\prec Y_i\,,\;\text{for all }\,i=1, \dots , n,
$$
and let $\fkg $ denote the associated notion of stochastic domination of probability measures. Recall \eqref{eq:B-y-meas}. 
{
	As it was  proved in \cite{corwinhammond} 
	\begin{lemma}\label{lem:SD-CH}
		For any $n$,$\ell , r$, $h_-\prec g_-$ and  $h_+ \prec g_+$ the following holds. 
		If,  $\ux\prec \uu$ and $\uy\prec \uv$, then 
		\be{eq:FKG-CH} 
		{\mathbf B}_{\ell , r}^{\ux, \uy}\left[\cdot   | h_- , h_+\right]\fkg 
		{\mathbf B}_{ \ell , r}^{\uu , \uv}\left[\cdot  | g_- , g_+\right].
		\ee 
	\end{lemma}	
	Our generalization \cite{CIW18} to measures with geometric area tilts reads as: 
}
\begin{lemma}\label{lem:SD}
	For any $n$,$\ell , r$, $h_-\prec g_-$, $h_+ \prec g_+$ and  $\urho\succ \ukappa$ the following holds. 
	If,  $\ux\prec \uu$ and $\uy\prec \uv$, then 
	\be{eq:FKG-2} 
	\bbP_{\ell , r}^{\ux, \uy}\left[\cdot   | h_- , h_+, \urho\right]\fkg 
	\bbP_{\ell , r}^{\uu , \uv}\left[\cdot  | g_- , g_+, \ukappa \right].
	\ee 
	Moreover, for an $n$-tuple $\uchi = \lbr \chi_1 ,\dots ,\chi_n\rbr$ of smooth boundary condition let $\uchi^\prime$ be the $n$-tuple of corresponding first derivatives. Then,
	\be{eq:FKG-1} 
	\bbP_{\ell , r}^{\uxi , \uzeta}\left[\cdot   | h_- , h_+, \urho\right]\fkg 
	\bbP_{ \ell , r}^{\unu , \ueta}\left[\cdot  | g_- , g_+, \ukappa \right], 
	\ee 
	whenever, 
	$h_-\prec g_-$, $h_+ \prec g^+$, $\urho\succ \ukappa$ and, 
	both $\underline\xi'
	\succ \unu'$ and $\underline\zeta'
	\succ \ueta'$. 
	In particular, \eqref{eq:FKG-1} holds if 
	$\uxi = \unu$ and $\uzeta = \ueta$ (by approximation without any assumptions on smoothness).
\end{lemma}
\subsection{Proof of Theorem~\ref{thm:tight-k} }
To facilitate the exposition we shall restrict attention to symmetric 
intervals  $[\ell , r] = [-T , T]\supseteq [-\gamma -2,\gamma +2]$. The choice of the constant $2$ here is purely conventional, and any positive number could be handled with minor modifications. 

Following the standard compactness criterion (see \cite[Theorem 8.10]{billingsley}), and thanks to the single time tightness statement that follows from Theorem  \ref{thm:max-control}, the proof of Theorem \ref{thm:tight-k} will be reduced to the control of the maximal modulus of continuity
\be{eq:mkg1}
m^{k,\gamma}(\uX,\delta):= \max_{1\leq i\leq k}\suptwo{s,t\in[-\gamma,\gamma]:}{|s-t|<\delta}|X_i(s)-X_i(t)|.
\ee
Thus, Theorem~\ref{thm:tight-k} is a consequence of the following technical estimate. 
\begin{proposition} 
\label{prop:mkg-control}
Fix $\gamma>0$, $k\in\bbN$. For any $\eta,\ep>0$, there exists $\delta>0$ such that 
\be{eq:mkg2}
\sup_{T\geq\gamma+2, \,n>k}
\mu^{k , \gamma}_{n ;T}\left(
m^{k,\gamma}(\uX,\delta)\geq \eta\right)\leq \ep
 \ee
The estimate above holds as well for the zero boundary distribution $ \mu^{0,k , \gamma}_{n ;T}$. 
\end{proposition} 

\subsection{Scheme of proof of Proposition \ref{prop:mkg-control}}
We fix $\gamma >0$, and $k\in \bbN$. All estimates to be derived below implicitly depend on these two parameters. On the other hand, they will be uniform in $n>k$ and $T\geq \gamma+2$. We give the details of the proof in the case of free boundary conditions only, since the case of  zero boundary conditions requires only cosmetic changes.

Consider 
the event $E=E(\delta,\eta)$ from \eqref{eq:mkg2}: 
\be{eq:mkgev}
E=\left\{
m^{k,\gamma}(\uX,\delta)\geq \eta\right\}.
\ee
Clearly, this event belongs to the $\sigma$-algebra $\sigma(X_i(t),\,t\in[-\gamma,\gamma],\,i=1,\dots,k)$.
Thus, we want to prove that 
\be{eq:mkgevtop}
\mu^{k , \gamma}_{n ;T}\left(E\right)=\bbP_{n;T}(E)\leq \ep.
\ee
We shall rely on the fact that the probability of the event $E$ can be made suitably small in the case of $k$ independent Brownian bridges, uniformly in the boundary conditions, provided the latter are taken in a bounded set. 
\begin{lemma} 
\label{lem:mkg-control_b}
For any $\eta,\ep>0$, and $M>0$, there exists $\delta>0$ such that 
\be{eq:mkg2b}
\sup_{\ux,\uy\in \bbA_k^+ (M)}
{{\bf \Gamma}}_{k, \gamma}^{\ux , \uy}\lb m^{k,\gamma}(\uX,\delta)\geq \eta  \rb \leq \ep 
 \ee
\end{lemma}
\begin{proof}
This follows by reduction to the case of $k$ independent Brownian bridges from 0 to 0, see e.g. \cite[p. 463]{corwinhammond}.  
\end{proof}
The proof of Proposition \ref{prop:mkg-control} will be achieved by controlling  the probability $\bbP_{n;T}(E)$ in terms of the probability $ {{\bf \Gamma}}_{k, \gamma}^{\ux , \uy}\lb E  \rb$, uniformly in $n>k, T\geq \gamma+2$. Our proof comprises several intermediate steps that can be roughly summarized as:

\bigskip 

\noindent
\step{1} Control of probabilities of large excursion events for top and bottom paths $X_1,X_{k+1}$. 
\smallskip 

\noindent
\step{2} Reduction to probabilities without area tilts.

\smallskip 

\noindent
\step{3} Use of heat kernel bounds for Brownian bridges in conical domains.

\smallskip 

\noindent
\step{4} 
Control of boundary values $X_{k}\lb\pm \gamma\rb$ 
and of $\min_{j=1}^{k-1}\lb X_j (\pm \gamma)- X_{j+1} (\pm\gamma)\rb$. 
\subsection{Control of maxima}
For each $M>0$, define the event
\be{eq:Event}
H^{i}(M) :=\lbr  
\max_{t\in [-\gamma - 2, \gamma +2]} X_{i} (t )\leq \lambda^{-(i-1)/3}M \rbr .
\ee
 The main estimates concerning large excursions  to be used below are gathered in the next lemma. Recall the notation \eqref{eq:B-y-meas}. In particular, ${\mathbf B}_{n; T}^{\ux , \uy } \left[  \, \cdot\, \big|\, 
	h
	\right]$ denotes the polymer measure with no area tilts and with floor $h$.   
 
\begin{lemma}\label{lem:contr_max}
The following holds for all fixed $k\in\bbN$ and $\gamma>0$:
\begin{enumerate}
\item For all $\ep>0$, there exists $L=L(k,\gamma,\ep)>0$ such that for all $M>0$, 
$\ux, \uy \in \bbA_k^+ ( M)$ and $h\prec M$, with $x_k>h(-\gamma-2)$ and $y_k>h(\gamma+2)$:
\be{eq:max1b}
{\mathbf B}_{k; \gamma+2}^{\ux , \uy } \left[ H^{1}(M+L)  \big|\, 
	h	\right] \geq 1-\ep. 
\ee
\item For all $i=1,\dots,n$, $\alpha\in(0,1/2)$ and $M>0$:
\be{eq:Markovk} 
\bbP_{n, T} \lb H^{i}(M) \big|~ a, \lambda \rb \geq 1- \frac{\xi^*_\alpha + (\gamma+2)^\alpha}{M}. 
\ee
where $\xi^*_\alpha$ is as in Theorem \ref{thm:max-control}.
\end{enumerate}
\end{lemma}
\begin{proof}
Using Lemma \ref{lem:SD-CH}  and a recursive argument (see e.g.\ \cite[Proposition 2.1]{CIW18}) , it is not hard to see that the maximum of $X_1$ under ${\mathbf B}_{k;T}^{\ux , \uy } \left[ \cdot \big|\, 
	h	\right]$ is stochastically dominated 
by $M$ plus the sum of $k$ independent copies of the maximum of the Brownian bridge of length $2T$ between $0$ and $0$ constrained to be nonnegative. 
This implies \eqref{eq:max1b}. 

Inequality \eqref{eq:Markovk} for $i=1$ is implied by Markov inequality and Theorem \ref{thm:max-control}. The case of general $i$ is obtained using Brownian scaling and the stochastic domination Lemma~\ref{lem:SD}, see \cite[Remark 2]{CIW18} for the details. 
\end{proof}

\subsection{Reduction to probabilities without area tilts} 
Estimate \eqref{eq:Markovk} applied to $i=1$ and $i=k+1$ implies that by taking $M$ sufficiently large (depending on $\ep$)
\be{eq:est1}
\bbP_{n;T}(E) \leq \ep/2 + \bbP_{n;T}\lb E\cap H^{1}(M)\cap H^{k+1}(M)\rb.
\ee
Recall the notation \eqref{eq:Om-h}. Since $E$ depends only on top $k$ paths, it follows that
\be{eq:est11}
\bbP_{n;T}\lb E\cap H^{1}(M)\cap H^{k+1}(M)\rb \leq \sup_{f\prec \lambda^{-k/3} M}\bbP_{n;T}\lb E\cap \Omega_{k, \gamma +2}^{f , {M}}~\Big| ~\Omega_{k, \gamma +2}^{f}\rb.
\ee
Cleary, given $f\prec \lambda^{-k/3}M$,  
\[ 
\bbP_{n, T} 
\lb  E \cap \Omega_{k, \gamma +2}^{f , {M}} 
~\Big| ~\Omega_{k, \gamma +2}^{f}\rb \leq 
\bbP_{n, T} 
\lb  E \cap \Omega_{k, \gamma +2}^{f , {M}} 
~\Big| ~  \uX (\pm (\gamma +2))\in {\bbA_{n}^+} (M ); \Omega_{k, \gamma +2}^{f , {2M}} \rb.
\]
 By the Brownian-Gibbs property, the last expression is bounded by
\[ 
\sup_{\ux, \uy \in \bbA_k^+ ( M)} \frac{
{\mathbf B}_{k, \gamma +2}^{\ux, \uy} 
\lb  \Phi(\uX);E \cap \Omega_{k, \gamma +2}^{f , {M}} \rb
}
{ {\mathbf B}_{k, \gamma +2}^{\ux, \uy} 
\lb  \Phi(\uX); \Omega_{k, \gamma +2}^{f , {2M}}\rb},
\]
where $\Phi(\uX)= {\rm e}^{-\sum_0^{k-1} a\lambda^{j}\calA_{\gamma+2} (X_i )}$.  In the numerator we estimate $\Phi(\uX)\leq 1$ while in the denominator we may estimate 
\[
\Phi(\uX)\geq e^{-2a{M}(2\gamma +4)\sum_0^{k-1} \lambda^j}=:
{\frac{1}{C(M)}}.
\] 
Taking $M$ appropriately large, \eqref{eq:max1b} shows that 
\[
{\mathbf B}_{k, \gamma +2}^{\ux, \uy} 
\lb   \Omega_{k, \gamma +2}^{f , {2M}} \rb 
\geq \frac{1}{2} 
{\mathbf B}_{k, \gamma +2}^{\ux, \uy} 
\lb   \Omega_{k, \gamma +2}^{f } \rb 
\]
uniformly in 
${f\prec \lambda^{-k/3} M}$ and  
${\ux, \uy \in \bbA_k^+ ( M)} $. 
Hence, we infer that {the probability term on the right hand side of}  \eqref{eq:est1} is bounded above by  
\be{eq:targBound3-1}
2	\,C(M)\sup_{f\prec \lambda^{-k/3} M}\ 
\sup_{\ux, \uy \in \bbA_k^+ ( M)} 
{\mathbf B}_{k, \gamma +2}^{\ux, \uy} 
\lb  E \cap \Omega_{k, \gamma +2}^{f , {M}} 
~\big|~ \Omega_{k, \gamma +2}^f\rb.
\ee
Above
it is understood that the 
conditional probability is taken to be zero if either $x_k < f (-\gamma -2)$ or $y_k < f (\gamma +2 )$. 
%

For simplicity of notation, we set $\ell=-\gamma-2$ and $r=\gamma+2$. Moreover, we call 
$\calF_M^k$ the family of 
 functions $f\geq 0$ on $[\ell , r]$ satisfying 
 $\sup_{t} f(t )\leq \lambda^{-k/3}M$.   
 Then, probabilities in \eqref{eq:targBound3-1} are
now written as 
 \be{eq:targBound4} 
 \sup_{f\in \calF_M^k }\ 
 \sup_{\ux, \uy \in \bbA_k^+ ( {M})} 
 {\mathbf B}_{k; \ell, r}^{\ux, \uy} 
 \lb  E \cap \Omega_{k; \ell, r}^{f , {M}}
 ~\big| ~ \Omega_{k; \ell, r}^f
 \rb.
 \ee
 Following Section~5 in \cite{corwinhammond} let us set 
 \be{eq:f1f}
 {
 \Omega_{k; \ell, r}^f =  
 \Omega_{k; \gamma}^f\cap 
 \Omega_{k; [\ell, r]\setminus [-\gamma ,\gamma]}^{f} ,
}
 \ee
 where we use the notation \eqref{eq:Om-h-g}.
 Then, 
 since $\Omega_{k; \ell, r}^{f , {M}} \subset \lbr X_1 (\pm \gamma) \leq M\rbr$, 
 \be{eq:bfB-bound}
 \begin{split}
 {\mathbf B}_{k; \ell, r}^{\ux, \uy} 
 \lb  E\cap \Omega_{k; \ell, r}^{f , {M}} ~\big| ~ \Omega_{k; \ell, r}^f\rb 
 &
 { 
 \leq {\mathbf B}_{k; \ell, r}^{\ux, \uy} 
 \lb  E;  X_1 (\pm \gamma) \leq M~\big| ~ 
 \Omega_{k; \gamma}^f
 \Omega_{k; [\ell, r]\setminus [-\gamma ,\gamma]}^{f}\rb
}
 \\
 &\leq \frac{{\mathbf B}_{k; \ell, r}^{\ux, \uy} 
 \lb  E~\big| ~ 
 X_1 (\pm \gamma) \leq M;\Omega_{k; [\ell, r]\setminus [-\gamma ,\gamma]}^{f}\rb
 }
{{\mathbf B}_{k; \ell, r}^{\ux, \uy} \lb
\Omega_{k; \gamma}^f ~\big|~
\Omega_{k; [\ell, r]\setminus [-\gamma ,\gamma]}^{f}	
\rb },
\end{split}
 \ee
 where we relied on elementary inequality 
 	\[
 \bbP (AB |CD) = \frac{\bbP (ABCD)}{\bbP (CD)} 
 	\leq 
 \frac{\bbP (A |B D) \bbP (D) }{\bbP (CD)} = 
  \frac{\bbP (A |B D)}{\bbP (C |D)} .
 \]
For the numerator in \eqref{eq:bfB-bound} we use the fact that  $E$ depends only on the time interval $[-\gamma ,\gamma]$, and the spatial Markov property for $k$ independent  Brownian bridges to obtain 
\be{eq:bfB-bound1}
{\mathbf B}_{k; \ell, r}^{\ux, \uy} 
 \lb  E~\big| ~ 
 X_1 (\pm \gamma) \leq M;\Omega_{k; [\ell, r]\setminus [-\gamma ,\gamma]}^{f}\rb
\leq \sup_{\uv, \uw \in \bbA_n^+ ( M) }
 {{\bf \Gamma}}_{{k, \gamma}}^{\uv, \uw} 
 \lb  E\rb.
 \ee
The last expression can be made arbitrarily small thanks to the estimate of  Lemma \ref{lem:mkg-control_b}.

In conclusion, to finish the proof of Proposition \ref{prop:mkg-control} we need to show
that for any $\gamma , M$ and $k$ there exists $\beta = \beta (k , M , \gamma ) >0$, such that 
\be{eq:targBound5} 
\inf_{f\in \calF_M^k }\  
\inf_{\ux, \uy \in \bbA_k^+ ( M)} 
{\mathbf B}_{k; \ell, r}^{\ux, \uy} \lb
\Omega_{k; \gamma}^f ~\big|~
\Omega_{k; [\ell, r]\setminus [-\gamma ,\gamma]}^{f}	
\rb 
\geq \beta .
\ee
Here  it is understood that $\ux,\uy\in\bbA_k^+ ( M)$
also satisfy $x_k>f(\ell)$, $y_k>f(r)$.

\subsection{{Main technical estimates about constrained  {Brownian bridges}}}
\label{sub:main}
{
	In the sequel $k\in\bbN$ is fixed and we employ the  reduced notation 
	$\Omega_{\ell, r}^{f}$,  
	$\Omega_{\gamma}^{f}$, 
	$\Omega_{[\ell, r]\setminus [-\gamma , \gamma]}^{f}$
	for the sets of  $k$-tuples of ordered  paths.
}

Let us take a look at our target \eqref{eq:targBound5}. 
By the spatial Markov property of Brownian bridges , 
\be{eq:decom-B}
{\mathbf B}_{\ell, r}^{\ux, \uy} \lb
\Omega_{ \gamma}^f ~\big|~
\Omega_{[\ell, r]\setminus [-\gamma ,\gamma]}^{f}\rb  = 
{\mathbf B}_{k; \ell, r}^{\ux, \uy} \lb
{\bf \Gamma}^{\uX (- \gamma ) , \uX (\gamma )}_{\gamma} (\Omega_{ \gamma}^f) 
 ~\big|~
\Omega_{[\ell, r]\setminus [-\gamma ,\gamma]}^{f}\rb.
\ee

\noindent 
\begin{definition} 
\label{def:Good-eps} 
Given $f\in \calF_M^k $ and $\delta >0$ let us define the set $\calG_{k, \gamma}^f (\delta )$ of pairs of boundary conditions 
$\lb\ux , \uy \rb \in \bbA_k^+\times \bbA_k^+$, such that 
\be{eq:Good-eps} 
{\mathbf B}^{\ux, \uy}_{\gamma} \lb \Omega_{\gamma}^f\rb \geq \delta . 
\ee	
	\end{definition}
{ 
In view of \eqref{eq:decom-B} our target \eqref{eq:targBound5} would follow if we show that 
there exist $\delta >0$ and $\chi >0$ such 
 the following estimate holds:
 \be{eq:Set-estimate} 
 \inf_{f\in \calF_M^k }\  
 \inf_{\ux, \uy \in \bbA_k^+ ( M)} 
 {\mathbf B}_{\ell, r}^{\ux, \uy} 
 \lb \big( \uX (-\gamma) , \uX (\gamma )\big)\in 
 \calG_{k, \gamma}^f (\delta )
 ~\big| ~ \Omega_{[\ell , r]\setminus [-\gamma ,\gamma]}^{f}\rb 
 \geq \chi .
 \ee
}
It remains to prove the following two lemmas: 
\begin{lemma} 
\label{lem:gaps} 
Consider sets 
\be{eq:gaps} 
\calS_k (\eta, L ) = \lbr  \ux~:~   x_k > 1+ M\lambda^{-k/3} ,x_1 <L  \rbr\bigcap_{j=1}^{k-1}\lbr x_j - x_{j+1} >\eta\rbr.
\ee	
Fix  $k\in \bbN$ and $\gamma >0$. Then  for any 
 $\eta  >0$ 
and $L > 1+ M\lambda^{-k/3} + k\eta $   there exists 
$\delta >0$
  such that 
\be{eq:inclusion-gaps} 
\calS_k (\eta, L )\times\calS_k (\eta, L ) \subset 
\calG_{k, \gamma}^f (\delta )
\ee
for any $f\in \calF_M^k $. 
\end{lemma}
{
\begin{lemma} 
	\label{lem:good_points-g}
	There exists $\chi >0, \eta>0$ and   $L > 1+ M\lambda^{-k/3} + k\eta$ such that 
	\be{eq:calS-steer-g} 
	\inf_{f\in \calF_M^k }\  
	\inf_{\ux, \uy \in \bbA_k^+ ( M)}
	{\mathbf B}_{\ell , r}^{\ux ,\uy}\lb \lb \uX (-\gamma), \uX (\gamma)\rb \in \calS_k (\eta, L )\times\calS_k (\eta, L )~| ~
	\Omega_{[\ell , r]\setminus [-\gamma ,\gamma]}^{f}
	\rb \geq \chi . 
	\ee	
\end{lemma}
}
\subsection{Proof of Lemma~\ref{lem:gaps} }
Since $f(t)\le \lambda^{-k/3}M$ for all $t\in[-\gamma,\gamma]$, and $\ux,\uy\in\calS_k (\eta, L )$:
$$
{\mathbf B}^{\ux, \uy}_{\gamma} \lb \Omega_{k ;\gamma}^f\rb \geq
{\mathbf B}^{\ux, \uy}_{\gamma} \lb \Omega_{k ;\gamma}^{\lambda^{-k/3}M}\rb.
$$
Let $p_{S,T}(\ux,\uz)$ denote the density of the measure ${\mathbf P}^{\ux}_{S,T} \lb \uX(T)\in d\uz;\,\Omega^+_{k,S,T} \rb$.
Then,  by the translation invariance and by the simmetry of the Brownian motion,
\begin{align*}
{\mathbf B}^{\ux, \uy}_{\gamma} \lb \Omega_{k ;\gamma}^{\lambda^{-k/3}M}\rb
&=\int_{\bbA^+_k}p_{0,\gamma}\lb\ux -\underline{\lambda^{-k/3}M},\uz\rb p_{0,\gamma}\lb\uy -\underline{\lambda^{-k/3}M},\uz\rb d\uz\\
&\ge\int_{\bbA^+_k(L)}p_{0,\gamma}\lb\ux - \underline{\lambda^{-k/3}M},\uz\rb p_{0,\gamma}\lb\uy -\underline{\lambda^{-k/3}M},\uz\rb d\uz.
\end{align*}  
{Above we use $\uc = \lb c, c,\dots , c\rb$ for the $k$-tupple with constant entries $c= \lambda^{-k/3}M$.}

{Fix $t_0 >0$. }
According to (0.3.2) in \cite{Varopoulos1999}, 
{
for any $L$  there exists a non-decreasing function $t\mapsto C(t,L)$ on 
$[t_0 , \infty)$
}
 such that
\begin{equation}
\label{Harnack_principle}
\frac{1}{C(t , L)}U(\ux)U(\uz)\le p_{0,t}(\ux,\uz)\le C(t, L)U(\ux)U(\uz),\quad 
x_1,z_1<L.
\end{equation}
where
$$
U(\ux):=\prod_{j=1}^k x_j\prod_{i<j}(x_i^2-x_j^2)
$$
is the positive  harmonic function in $\bbA^+_k$.  It is then clear that if $\ux\in \calS_k (\eta, L )$ then
$$
U\lb\underline{x}-\underline{\lambda^{-k/3}M}\rb\ge (2\eta)^{k(k-1)/2}.
$$
Therefore,
$$
{\mathbf B}^{\ux, \uy}_{\gamma} \lb \Omega_{k ;\gamma}^f\rb \geq
\frac{1}{C^2(\gamma, L)}(2\eta)^{k(k-1)}\int_{\bbA^+_k(L)}U^2(\uz)d\uz
$$
and the proof is complete.
\subsection{Proof of Lemma~\ref{lem:good_points-g}}
\label{sub:lem2.6}
The proof comprises two steps: 

\bigskip
\noindent
\step{1} {\bf Control of $X_{k}\lb\pm (\gamma +1)\rb$.} 
One first checks that there exists $\epsilon = \epsilon (\gamma, M, k ) >0$, such that for all $f\in \calF_M^k $, 
\be{eq:chechXk} 
\inf_{\ux, \uy \in \bbA_k^+ ( M)} 
{\mathbf B}_{\ell, r}^{\ux, \uy} 
\lb {L> X_1 (\pm (\gamma +1))>} X_k (\pm (\gamma +1)) \geq  1+M\lambda^{-k/3} ~\big| ~ \Omega_{[\ell,r]\setminus [-\gamma , \gamma]}^{f}\rb 
\geq \epsilon\,.
\ee
Let us observe that if $L$ is sufficiently large, then the bound 
\eqref{eq:chechXk} is a consequence of
\be{eq:chechXk1} 
\inf_{\ux, \uy \in \bbA_k^+ ( M)}
{\mathbf B}_{\ell, r}^{\ux, \uy} \lb 
{ X_1 (\pm (\gamma +1))>}
X_k (\pm (\gamma +1)) \geq 1+ M\lambda^{-k/3} ~\big| ~
{ 
	\Omega_{[\ell,r]\setminus [-\gamma , \gamma]}^{{0}}
}
\rb 
\geq \epsilon .
\ee
Indeed, set $A = \lbr L> X_1 (\pm (\gamma +1))\rbr$ and 
$B =\lbr X_k (\pm (\gamma +1)) \geq 1+ M\lambda^{-k/3}\rbr$. 
Both $A^{\sfc}$ and $B$ are monotone non-decreasing. Hence, Lemma~\ref{lem:SD-CH} implies that  
\be{eq:mon-chain} 
{\mathbf B}_{\ell, r}^{\ux, \uy} 
\lb AB ~\big| ~ \Omega_{[\ell,r]\setminus [-\gamma , \gamma]}^{f}\rb
 \geq {\mathbf B}_{\ell, r}^{\ux, \uy} 
\lb B ~\big| ~ \Omega_{[\ell,r]\setminus [-\gamma , \gamma]}^{0}\rb 
- 
{\mathbf B}_{\ell, r}^{\ux, \uy} 
\lb A^{c} ~\big| ~ \Omega_{[\ell,r]\setminus [-\gamma , \gamma]}^{M\lambda^{-k/3}}\rb .
\ee
Moreover, 
for any $\epsilon>0$,
\begin{equation}
\label{eq:chechX1}
{\mathbf B}_{\ell, r}^{\ux, \uy} \lb
A^c\,\big|\, 
{\Omega_{[\ell,r]\setminus [-\gamma , \gamma]}^{M\lambda^{-k/3}}}
\rb<{\epsilon/2}
\end{equation}
for all $L$ large enough, uniformly in $\ux,\uy\in\bbA_k^+(M)$. 
The estimate \eqref{eq:chechX1} follows from \eqref{eq:max1b} and the observation that, by monotonicity \eqref{eq:FKG-CH}, 
\[
{\mathbf B}_{\ell, r}^{\ux, \uy} \lb
X_1 (\pm (\gamma +1)) \geq L\,\big|\, 
{\Omega_{[\ell,r]\setminus [-\gamma , \gamma]}^{M\lambda^{-k/3}}
} 
\rb
\leq 
{\mathbf B}_{\ell, r}^{\ux, \uy} \lb
X_1 (\pm (\gamma +1)) \geq L\,\big|\, 
{\Omega_{[\ell , r]}^{M\lambda^{-k/3}}
} 
\rb.
\]
The above shows that \eqref{eq:chechXk1} implies \eqref{eq:chechXk} with $\epsilon$ replaced by $\epsilon/2$. 
 
 \smallskip

	\eqref{eq:chechXk1}  will be derived in Subsection~\ref{sub:refined-X} below. 
	 \qed
\bigskip 

\noindent\step{2}
Once  bottom boundary conditions at $\pm (\gamma +1)$ satisfy $x_k, y_k \geq  1+ M\lambda^{-k/3} $ 
{
and
$x_1,y_1<L$, 
}
 one can steer into
$\calS_k (\eta, L )\times\calS_k(\eta, L )$ at $\pm\gamma$. 
\begin{lemma} 
\label{lem:good_points}
	Recall \eqref{eq:f1f}. Fix ${L}$ and $\eta$ 
	{such that $L > 1+ M\lambda^{-k/3} + k\eta $, in particular such that $\calS_k (\eta, L )$ has a positive Lebesgue measure.}
There exists $\chi = \chi (k ,L , \eta ) >0 $ such that 
\be{eq:calS-steer} 
{\mathbf B}_{\gamma +1}^{\ux ,\uy}\lb \lb \uX (-\gamma), \uX (\gamma)\rb \in \calS_k (\eta, L )\times\calS_k (\eta, L )~| ~
{
\Omega_{[-\gamma-1 , \gamma+1]\setminus [-\gamma , \gamma]}^{f}
}
 \rb \geq \chi 
\ee	
for all $f\in \calF_M^k $ and for all $\ux , \uy \in \bbA_k^+$ satisfying $x_k , y_k \geq  1+ M\lambda^{-k/3}$ and
$x_1,y_1<L$. 
\end{lemma}
\begin{proof}
Set, for brevity
$$
{
A = A(\eta , L)
}
:=\left\{
\lb \uX (-\gamma), \uX (\gamma)\rb \in \calS_k (\eta, L )\times\calS_k (\eta, L )
\right\}\ 
\text{and {$\widehat\Omega_{\gamma}^f := 
	\Omega_{[-\gamma-1 , \gamma+1]\setminus [-\gamma , \gamma]}^{f}$}}
$$
Then, {since $0\leq f \leq M\lambda^{-k/3}$}, 
\begin{align}
\label{eq:lem_2.2_1}
{\mathbf B}_{\gamma +1}^{\ux ,\uy}\left(A\big|
{\widehat\Omega_{\gamma}^f}
\right)
&=\frac{{\mathbf B}_{\gamma +1}^{\ux ,\uy}\left(A ;
{\widehat\Omega_{\gamma}^f	}
	\right)}
{{\mathbf B}_{\gamma +1}^{\ux ,\uy}\left(
\widehat\Omega_{\gamma}^f	
\right)}
\ge \frac{{\mathbf B}_{\gamma +1}^{\ux ,\uy}\left(A ; 
	\widehat\Omega_{\gamma}^{M\lambda^{-k/3}}
	\right)}
{{\mathbf B}_{\gamma +1}^{\ux ,\uy}\left(
\widehat\Omega_{\gamma}^0	
\right)}.
\end{align}
{
	It is straightforward to check that  \eqref{Harnack_principle}   implies the following 
upper bound on ${\mathbf B}_{\gamma +1}^{\ux,\uy}\left(    \widehat\Omega_{\gamma}^0\right)$: There exists 
$\widehat{C} = \widehat{C}(\gamma , L)$ such that 
\be{eq:UB-Om-hat} 
\tfrac{U(\ux)U(\uy)}{\widehat{C}(\gamma , L)} \leq 
{\mathbf B}_{\gamma +1}^{\ux,\uy}\left(    \widehat\Omega_{\gamma}^0\right)
\le \widehat{C}(\gamma , L)U(\ux)U(\uy), 
\ee
 uniformly in 
$\ux , \uy \in \bbA_k^+ (L)$. 
}
Furthermore, for every $\ux$ with $x_1<L$ we have
$$
U(\ux)<L^k(2L)^{k(k-1)/2}\prod_{i<j}(x_i-x_j).
$$
This implies that
{
\begin{equation}
\label{eq:lem_2.2_2}
{\mathbf B}_{\gamma +1}^{\ux ,\uy}\left(
\widehat\Omega_{\gamma}^0
\right)
\le \widehat{C}(\gamma , L)(2L)^{k(k+1)}\prod_{i<j}(x_i-x_j)\prod_{i<j}(y_i-y_j).
\end{equation}
}

For the probability in the numerator in  the right-most 
expression in  \eqref{eq:lem_2.2_1} we have 
\begin{align}
\label{eq:DenomForA}
{\mathbf B}_{\gamma +1}^{\ux ,\uy}\left(A ; 
\widehat{\Omega}_{\gamma}^{\lambda^{-k/3}M}\right)
=\int_{\calS_k (\eta, L )}\int_{\calS_k (\eta, L )} p_{0,1}(\underline{\tilde x},\underline{\tilde u})
{q_{2\gamma}}
(\underline{\tilde u},\underline{\tilde v}) p_{0,1}(\underline{\tilde v},\underline{\tilde y})
d\underline{u}d\underline{v},
\end{align}
where $\underline{\tilde z}:=\underline{z}-\underline{\lambda^{-k/3}M}$,
$q_{2\gamma} (\underline{\tilde u},\underline{\tilde v})$ is the unconstrained  quantity defined in 
\eqref{eq:q-uxuy02}, and $p_{0,1}$ is the density defined before \eqref{Harnack_principle}. 

Since $\underline{u}$, $\underline{v}$ are in $\calS_k (\eta, L ){\subset \bbA_k (L )} $, 
$$
q_{2\gamma}(\underline{\tilde u},\underline{\tilde v})\ge
\tfrac1{\lb 4\pi\gamma\rb^{k/2}}\,{\rm e}^{-  \frac{kL^2}{4\gamma}}.
$$ 
as it readily follows from \eqref{eq:q-xy}.
Consequently,
\begin{align*}
{\mathbf B}_{\gamma +1}^{\ux ,\uy}\left(A ;
\widehat\Omega_{\gamma}^{M\lambda^{-k/3}}
\right)
\ge \frac{(2\eta)^{k(k-1)}{{\rm e}^{-  \frac{kL^2}{4\gamma}}}}{{\lb 4\pi\gamma\rb^{k/2}}}
\int_{\calS_k (\eta, L )} p_{0,1}(\underline{\tilde x},\underline{\tilde u})d\underline{u}
\int_{\calS_k (\eta, L )}p_{0,1}(\underline{\tilde v},\underline{\tilde y})d\underline{v}.
\end{align*}
Using \eqref{Harnack_principle} once again, we obtain
\begin{align*}
\int_{\calS_k (\eta, L )} p_{0,1}(\underline{\tilde x},\underline{\tilde u})d\underline{u}
\ge\frac{1}{C(\gamma, L)}U(\underline{\tilde x})\int_{\calS_k (\eta, L )}U(\underline{\tilde u})d\underline{u}.
\end{align*}
We have already shown that $U(\underline{\tilde u})\ge (2\eta)^{k(k-1)/2}$ for $\underline{u}\in\calS_k (\eta, L )$.
Furthermore, if  $x_k>1+\lambda^{-k/3}M$ then
$$
U(\underline{\tilde x})\ge \prod_{i<j}(x_i-x_j).
$$
Therefore
$$
\int_{\calS_k (\eta, L )} p_{0,1}(\underline{\tilde x},\underline{\tilde u})d\underline{u}
\ge\frac{1}{C(\gamma, L)} (2\eta)^{k(k-1)/2} \prod_{i<j}(x_i-x_j)
$$
and, consequently,
\begin{align}
\label{eq:lem_2.2_3}
{\mathbf B}_{\gamma +1}^{\ux ,\uy}\left( A; \widehat{\Omega}_{ \gamma +1}^{M\lambda^{-k/3}}\right)
\ge 
{ 
\frac{(2\eta)^{k(k-1)}{{\rm e}^{-  \frac{kL^2}{4\gamma}}}}{{\lb 4\pi\gamma\rb^{k/2}}
C^{{2}}(\gamma, L)}
}
\prod_{i<j}(x_i-x_j)\prod_{i<j}(y_i-y_j).
\end{align}
Plugging \eqref{eq:lem_2.2_2} and \eqref{eq:lem_2.2_3} into \eqref{eq:lem_2.2_1}, we get the desired estimate.
\end{proof}
\bibliographystyle{plain}

\subsection{
	{Proof of \eqref{eq:chechXk1}}}
\label{sub:refined-X}
By the Markov property,  
\eqref{Harnack_principle} 
{ and 
\eqref{eq:UB-Om-hat} 
},
\begin{align*}
&{\mathbf B}_{\ell, r}^{\ux, \uy} \lb
X_k (\pm (\gamma +1)) \geq 1+ M\lambda^{-k/3}, 
{
\Omega_{[\ell,r]\setminus [-\gamma , \gamma]}^{{0}}
}
\rb\\
&\ge\int_{ \calS_k(\eta, {2M})\times \calS_k(\eta, {2M})}p_{\ell,-\gamma-1}(\ux,\underline{u})
{
{\mathbf B}_{\gamma +1}^{\uu,\uv}\left(    \widehat\Omega_{\gamma}^0\right)
}
p_{\gamma+1,r}(\underline{v},\uy)d\underline{u}d\underline{v}\\
&\ge
{
\frac{1}{C^2({2}, {2M}) \widehat{C}(\gamma , {2M})}
}
 U(\ux)U(\uy)
\int_{ \calS_k(\eta, {2M})\times \calS_k(\eta, {2M})}U^{{2}}(\underline{u})U^{{2}}(\underline{v})d\underline{u}d\underline{v}\\
&\, {=:} \,C_1(\gamma,\eta, {M})U(\ux)U(\uy).
\end{align*}
Furthermore, {using upper bounds in 
	\eqref{Harnack_principle} 
	 and 
		\eqref{eq:UB-Om-hat}, we infer that 
		there exists $C_2(\gamma,M)$ such that
$$
{\mathbf B}_{\ell, r}^{\ux, \uy} \lb 
\Omega_{[\ell , r]\setminus [-\gamma , \gamma]}^{{0}}
\rb = 
\int p_{\ell,-\gamma-1}(\ux,\underline{u})
{
	{\mathbf B}_{\gamma +1}^{\uu,\uv}\left(    \widehat\Omega_{\gamma}^0\right)
}
p_{\gamma+1,r}(\underline{v},\uy)d\underline{u}d\underline{v}
\le C_2(\gamma,M)U(\ux)U(\uy).
$$
}
{
Therefore,  
\be{eq:chechXk1-1} 
\inf_{\ux, \uy \in \bbA_n^+ ( M)}
{\mathbf B}_{\ell, r}^{\ux, \uy} \lb 
X_k (\pm (\gamma +1)) \geq 1+ M\lambda^{-k/3} ~\big| ~
{ 
	\Omega_{[\ell, r]\setminus [-\gamma , \gamma]}^{{0}}
}
\rb 
\geq \chi_1
\ee
 holds with $\chi_1 :=C_1(\gamma,\eta, {M})/C_2(\gamma, M)$.
}
This concludes  the proof of  \eqref{eq:chechXk1}. 

\section{Limiting line ensemble}
\label{sec:line-ens}
In this section we prove Theorem \ref{thm:coarse1}  and Theorem \ref{thm:coarse2}.
The key technical tool will be a uniform control of the minimal gaps between lines.
\subsection{Minimal gaps}
As in \cite{corwinhammond}, for the purpose of analyzing the limiting line ensemble, an important quantity to be controlled is the minimal
gap between the paths, defined as
\be{eq:gkg1}
g^{k,\gamma}(\uX):= \min_{1\leq i\leq k-1}\inf_{s\in[-\gamma,\gamma]}|X_{i+1}(s)-X_i(s)|.
\ee
The main  result concerning this quantity is the following   
\begin{proposition} 
\label{prop:gkg-control}
For any $\ep>0$, there exists $\delta>0$ such that 
\be{eq:gkg2}
\sup_{T\geq\gamma+2,\,n>k}
\mu^{k , \gamma}_{n ;T}\left(
g^{k,\gamma}(\uX)\leq \delta\right)\leq \ep
 \ee
The estimate above holds as well for the zero boundary distribution $ \mu^{0,k , \gamma}_{n ;T}$. 
\end{proposition} 
We remark that  \eqref{eq:gkg2} shows that any limiting point
 of $\{\mu^{k , \gamma}_{n ;T}, \,T\geq \gamma+2, \,n>k\}$, must be concentrated on $\sfC\lb [-\gamma , \gamma ]; \bbA_k^+\rb$.  Since this holds for all $\gamma$ and $k$, it follows that if $\bbP$ is a limiting point of 
 the free or the zero boundary measures 
 as in Theorem \ref{thm:coarse1}, then $\bbP$ must be concentrated on $\sfC\lb \bbR, \bbA_\infty^+\rb$, that is $\bbP$ is non-intersecting.

\subsection{Proof of Proposition \ref{prop:gkg-control}}
In view of   the strategy which we have already employed to prove Proposition \ref{prop:mkg-control}, 
we will first reduce the problem to the case of Brownian measures without area tilts. Once this is accomplished, one could adapt the proof in \cite[Proposition 5.6]{corwinhammond} to obtain the desired result. However, below we will actually provide an alternative, perhaps more direct proof of the same estimate.  We provide the proof in the case of free boundary conditions only, since  the same argument applies with no modifications to the case of zero boundary conditions. 



Let $E=E(\delta, k , \gamma)$ denote the event 
\be{eq:E-gaps}
E = \{g^{k,\gamma}(\uX)\leq \delta\}.
\ee
As in \eqref{eq:est1}-\eqref{eq:est11} we may restrict ourselves to an upper bound on 
\[
\sup_{f\prec \lambda^{-k/3} M}\bbP_{n;T}\lb E\cap \Omega_{k, \gamma +2}^{f , {M}}\rb.
\]
Repeating the argument leading to \eqref{eq:targBound4}, it suffices to show that for all $\ep>0, M>0$ there exists $\delta>0$ such that 
 \be{eq:targBound40} 
 \sup_{f\in \calF_M^k }\ 
 \sup_{\ux, \uy \in \bbA_k^+ ( {M})} 
 {\mathbf B}_{k; \ell, r}^{\ux, \uy} 
 \lb  E \cap \Omega_{k; \ell, r}^{f , {M}}
 ~\big| ~ \Omega_{k; \ell, r}^f
 \rb \leq \ep,
 \ee
 where $\ell=-\gamma-2$ and $r=\gamma+2$. 

 	
	We now prove \eqref{eq:targBound40}. Fix $\gamma >0$, 
and $f\in \calF_M^k $. 
 	Fix boundary conditions  $\ux,\uy\in\bbA_k^+ ( M)$
 	satisfying  $x_k>f(\ell)$, $y_k>f(r)$. 
 	 By \eqref{eq:targBound5}, 
 	\be{eq:A-lb} 
 	{\mathbf B}_{k; \ell, r}^{\ux, \uy} \lb
 	\Omega_{k; \ell , r}^f\rb 
 	\geq \beta 
 	{\mathbf B}_{k; \ell, r}^{\ux, \uy}\lb 
 	\Omega_{k; [\ell, r]\setminus [-\gamma ,\gamma]}^{f}	
 	\rb 
 	\ee
 	Thus, it would be enough to check that  regardless of 
 	$f, \ux, \uy$ fixed as above, 
 	\be{eq:A-ub} 
 	{\mathbf B}_{k; \ell, r}^{\ux, \uy} \lb
 	\Omega_{k; \ell , r}^{f , M}\cap E \rb 
 	\leq \nu 
 	{\mathbf B}_{k; \ell, r}^{\ux, \uy}\lb 
 	\Omega_{k; [\ell, r]\setminus [-\gamma ,\gamma]}^{f}\rb 
 	\ee
 	where 
 	$E=E(\delta, k , \gamma)$ are  the events defined in 
 	\eqref{eq:E-gaps}, whereas $\nu = \nu (\delta, k , \gamma)$ 
 	satisfies: 
 	\be{eq:A-nu} 
 	\lim_{\delta\to 0} \nu (\delta, k , \gamma) = 0 .
 	\ee
 	Now, by \eqref{eq:calS-steer-g} there exists $\xi = \xi (k , \gamma ) > 0 $, such that 
 	\be{eq:A-lb-P} 
 	{\mathbf B}_{k; \ell, r}^{\ux, \uy}\lb 
 	\Omega_{k; [\ell, r]\setminus [-\gamma ,\gamma]}^{f}\rb 
 	\geq 
 	\xi \, 
 	{\mathbf P}_{k; \ell }^{\ux} \lb \Omega_{k; \ell  ,-\gamma}^{f}\rb 
 	\widehat{{\mathbf P}}_{k ; r}^{\uy} 
 	\lb \Omega_{k; \gamma ,r}^{f}\rb , 
 	\ee
 	where ${\mathbf P}_{k; \ell }^{\ux}$ denotes the measure associated to $k$ independent Brownian motions started at $\ux$ at time $\ell$ and $\widehat{{\mathbf P}}_{k ; r}^{\uy}$ is the 
	notation for the time-reversed $k$-tupple of independent Brownian motions started at 
 	time $r$ at $\uy$. 
 	
 	On the other hand, 
 	\be{eq:A-up-not} 
 	{\mathbf B}_{k; \ell, r}^{\ux, \uy} \lb
 	\Omega_{k; \ell , r}^{f , M}\cap E \rb \leq  
 	{\mathbf P}_{k; \ell }^{\ux} \lb \Omega_{k; \ell  ,-\gamma}^{f}\rb 
 	\widehat{{\mathbf P}}_{k ; r}^{\uy} 
 	\lb \Omega_{k; \gamma ,r}^{f}\rb 
 	\max_{\uu , \uv \in \bbA_k^+ (M )} 
 	{\mathbf B}_{k ;\gamma}^{\uu , \uv} \lb \Omega_{k; \gamma}^{0 }\cap E\rb
 	\ee
	If we show that  
 	\be{eq:A-up-E} 
 	\max_{\uu , \uv \in \bbA_k^+ (M )} 
 	{\mathbf B}_{k ;\gamma}^{\uu , \uv} \lb \Omega_{k; \gamma}^{0 }\cap E (\delta , k , \gamma )\rb \leq \nu  (\delta , k , \gamma ) , 
 	\ee
 	with $\nu$ satisfying \eqref{eq:A-nu}, then \eqref{eq:A-ub} will hold with $\nu^\prime = \nu/\xi$.
 	
 	We first note that 
 	\begin{align*}
 	{\mathbf B}_{k ;\gamma}^{\uu , \uv} \lb \Omega_{k; \gamma}^{0 }\cap E (\delta , k , \gamma )\rb
 	&\le\sum_{j=1}^{k-1}{\mathbf B}_{k ;\gamma}^{\uu , \uv} 
 	\lb \inf_{|t|\le\gamma}(X_j(t)-X_{j+1}(t))\in(0,\delta)\rb 	
 	\end{align*}
    Then, by the scaling property of the Brownian motion,
    \begin{align*}
    \max_{\uu , \uv \in \bbA_k^+ (M )} 
 	{\mathbf B}_{k ;\gamma}^{\uu , \uv} \lb \Omega_{k; \gamma}^{0 }\cap E (\delta , k , \gamma )\rb \leq (k-1)
 	\max_{x,y>0} {\mathbf B}_{2\gamma}^{x, y} 
 	\lb \inf_{|t|\le2\gamma}X(t)\in(0,
	\delta) \rb
    \end{align*}
    If $x\le\sqrt{2}\delta$ then
    \begin{align*}
    {\mathbf B}_{2\gamma}^{x, y} 
 	\lb \inf_{|t|\le2\gamma}X(t)\in(0,\delta) \rb
 	&\le \int_0^\infty {\mathbf P}_{-2\gamma,0}^{x}
 	\lb\inf_{t\in[-2\gamma,0]}X(t)>0, X(0)\in dz\rb
 	{\mathbf B}_{0,2\gamma}^{z, y}(1)\\
 	&\le\frac{1}{2\sqrt{\pi\gamma}}{\mathbf P}_{-2\gamma,0}^{x}
 	\lb\inf_{t\in[-2\gamma,0]}X(t)>0\rb
	\\&
 	\le\frac{x}{\sqrt 2\pi\gamma}\le\frac{\delta}{\pi\gamma},
    \end{align*}
    in the last step we have used the reflection principle.
    By the symmetry,
    \begin{align*}
    {\mathbf B}_{2\gamma}^{x, y} 
 	\lb \inf_{|t|\le2\gamma}X(t)\in(0,\delta) \rb
 	\le \frac{\delta}{\pi\gamma},
    \end{align*}
    for all $y\le\sqrt{2}\delta$.
    
    Consider now the remaining case $x,y>\sqrt{2}\delta$. Set 
    $$
    \tau=\inf\{t:X(t)=\sqrt{2}\delta\}.
    $$
    Then we have 
    \begin{align*}
    {\mathbf B}_{2\gamma}^{x, y} 
    \lb \inf_{|t|\le2\gamma}X(t)\in(0,\delta) \rb
    &={\mathbf B}_{2\gamma}^{x, y} 
 	\lb \inf_{|t|\le2\gamma}X(t)>0,\tau<2\gamma\rb\\
 	&={\mathbf B}_{2\gamma}^{x, y} 
 	\lb \inf_{|t|\le2\gamma}X(t)>0,\tau<0\rb
 	+{\mathbf B}_{2\gamma}^{x, y} 
 	\lb \inf_{|t|\le2\gamma}X(t)>0,\tau\in(0,2\gamma)\rb.
    \end{align*}
    By the strong Markov property and by the reflection principle,
    \begin{align*}
    {\mathbf B}_{2\gamma}^{x, y} 
 	\lb \inf_{|t|\le2\gamma}X(t)>0,\tau<0\rb
 	&=\int_{-2\gamma}^0 {\mathbf P}_{-2\gamma,0}^{x}\lb\tau\in ds\rb
 	{\mathbf B}_{s,2\gamma}^{\sqrt{2}\delta, y} 
 	\lb \inf_{t\in(s,2\gamma)}X(t)>0\rb\\
 	&\le\frac{2\delta}{\pi\gamma}{\mathbf P}_{-2\gamma,0}^{x}\lb\tau<0\rb.
    \end{align*}
    Therefore, by the symmetry,
    \begin{align*}
    {\mathbf B}_{2\gamma}^{x, y} 
 	\lb \inf_{|t|\le2\gamma}X(t)>0,\tau\in(0,2\gamma)\rb
 	\le\frac{2\delta}{\pi\gamma}.
    \end{align*}
    As a result,
    $$
    {\mathbf B}_{2\gamma}^{x, y} 
    \lb \inf_{|t|\le2\gamma}X(t)\in(0,\delta) \rb
    \le \frac{4\delta}{\pi\gamma}. 
    $$
    This completes the proof of \eqref{eq:targBound40}. 

\subsection{Proof of Theorem \ref{thm:coarse1}}
As already noted, the fact that $\bbP$ is non-intersecting follows from Proposition \ref{prop:gkg-control}. It is also immediate to check that the estimate on the maxima stated in the theorem follows from the control of  maxima established in Lemma \ref{lem:contr_max}. Notice that Lemma \ref{lem:contr_max} applies to the zero boundary case as well, as a consequence of the stochastic domination from Lemma \ref{lem:SD}.

Once we have the control of the minimal gaps from Proposition \ref{prop:gkg-control}, to prove the Brownian-Gibbs property of the limiting line ensemble we may use exactly the same coupling argument of \cite[Proposition 3.7]{corwinhammond}. Indeed, it is not hard to check that all the basic properties of Brownian bridges used in that argument can be extended  with only minor modifications to the case of Brownian bridges with area tilts. 

{ 
	The absolute continuity statement 
	 \eqref{eq:abs-cont} 
	follows by a literal repetition of arguments employed for the proof of 
Proposition~\ref{prop:mkg-control}
}

\subsection{Proof of Theorem \ref{thm:coarse2}}
In view of Theorem \ref{thm:tight-k}, we need only show convergence of finite dimensional distributions of
 $\bbP^0_{n:\ell,r}$ as $n\to\infty,\ell\to-\infty$, and $r\to\infty$. To this end, fix $m\in\bbN$, let $\calS=\{s_1,\dots,s_m\}\subset \bbR^m$, $\calI=\{i_1,\dots,i_m\}\in\bbN^m$, and let $\calT=\{t_1,\dots,t_m\}\in \bbR_+^m$. Consider the event
\[
E(\calS,\calI,\calT)=\{\uX\in\Omega:\; X_{i_j}(s_j) > t_j\,,\; j=1,\dots,m\}.
\]
It suffices to show that for each choice of $\calS,\calI,\calT$ there exists $u(\calS,\calI,\calT)$ such that
\be{eq:fdconv} 
 \bbP^0_{n:\ell,r}(E(\calS,\calI,\calT))\rightarrow u(\calS,\calI,\calT)\,
 \ee
as $n\to\infty,\ell\to-\infty$, and $r\to\infty$, regardless of the order of the limits. 
The main observation here is that, since the event $E(\calS,\calI,\calT)$ is increasing, the Brownian-Gibbs property on sub-intervals of $[\ell,r]$ and the stochastic domination from Lemma \ref{lem:SD} imply monotonicity  of the above probabilities:
\be{eq:fdmon} 
 \bbP^0_{n:\ell,r}(E(\calS,\calI,\calT))\leq  \bbP^0_{n',\ell',r'}(E(\calS,\calI,\calT))\,,
 \ee
 whenever $n'\geq n$, $r'\geq r$ and $\ell'\leq \ell$. Clearly, \eqref{eq:fdmon} implies \eqref{eq:fdconv}, and the proof of convergence is complete. 
 
 Concerning time-translation invariance, let $\calS_t=\{s_1+t,\dots,s_m+t\}$, $t\in\bbR$. Then, 
 \be{eq:fdconva} 
 \bbP^0_{n:\ell,r}(E(\calS_t,\calI,\calT)) = \bbP^0_{n:\ell-t,r-t}(E(\calS,\calI,\calT)) \,.
 \ee
Monotonicity \eqref{eq:fdmon} implies that 
 any two  sequences $(n^{(1)}_k,\ell^{(1)}_k,r^{(1)}_k), (n^{(2)}_k,\ell^{(2)}_k,r^{(2)}_k)$, such that $n^{(i)}_k\leq n^{(i)}_{k+1}$, $\ell^{(i)}_k\geq \ell^{(i)}_{k+1}$, $r^{(i)}_k\leq r^{(i)}_{k+1}$, $(n^{(i)}_k,\ell^{(i)}_k,r^{(i)}_k)\to (\infty,-\infty,\infty)$, $i=1,2$, must satisfy
 \be{eq:fdconvab} 
\lim_{k\to\infty} \bbP^0_{n^{(1)}_k;\ell^{(1)}_k,r^{(1)}_k}(E(\calS,\calI,\calT)) = \lim_{k\to\infty} \bbP^0_{n^{(2)}_k;\ell^{(2)}_k,r^{(2)}_k}(E(\calS,\calI,\calT))=u(\calS,\calI,\calT) \,.
 \ee
Therefore $\bbP^0_{n:\ell-t,r-t}(E(\calS,\calI,\calT))$ and $\bbP^0_{n:\ell,r}(E(\calS,\calI,\calT))$ have the same limit  
 \be{eq:fdconvt} 
 u(\calS_t,\calI,\calT)=u(\calS,\calI,\calT)\,,\qquad t\in\bbR,
 \ee
so that $\bbP^0$ is  invariant by time translations.

\subsection{{Directions of further research and open questions.}} 
\label{sub:open}
{
In this concluding Subsection we briefly summarize   open 
problems and future research directions. 
\subsubsection{Unicity, time-ergodicity  and structure of limiting line ensembles.} 
As it has been already mentioned 
Theorem~\ref{thm:coarse1} does not imply unicity of limiting 
line ensembles. There are two possible venues to address this issue,  
both require insights. Indeed, although there are no ready  determinantal formulas for finite dimensional distributions at hand (and hence, unlike \cite{corwinhammond} one cannot characterize the unique limiting ensemble from tightness considerations),  the simple geometric structure of tilts gives hope that there is an appropriate algebraic characterization to be uncovered. On the more analytic part, 
although mixing estimates of \cite{ioffevelenikwachtel} do not
directly apply in the case of random floors, one may hope that 
there is an appropriate adjustment, which would also imply ergodicity 
without resorting to exact solutions as in 
\cite{corwin2014ergodicity}. 
}

\subsubsection{{Scaling limits as $\lambda\downarrow 1$.}} 
{
It would be interesting to explore, how this scaling regime is related to the scaling regimes described in \cite{Bornemann}. 
}

\subsubsection{{Random walks and level lines ensembles}.}
{As it was stated in the very beginning, the original motivation was to understand the large scale structure of level lines of low temperature $(2+1)$-dimensional SOS interfaces above a hard wall \cite{caputoetal2014,caputoetal2016,IV2016,lacoin2017wetting}.  
	A natural intermediate step is to study ensembles of ordered random walks under
properly normalized  geometric area tilts, and to recover limiting line ensembles in the 
 $1:2:3$ diffusive rescaling.  In view of unbounded number of paths and in view of geometrically growing tilts such an endeavor clearly requires insights beyond \cite{ioffevelenikwachtel}. In the case of level lines one has to employ renormalization procedures of e.g. 
 \cite{ioffe2008ballistic,ioffeshlosmantoninelli}  for an effective finite scale description of level lines in terms of random walks. 
 Growing  number of macroscopic level lines  and multi-body interactions between them pose additional challenges  \cite{ioffeshlosmantoninelli,IS2017}. }

\bibliography{bib_tightness}
\end{document}